\documentclass{article}
\usepackage{amsmath, amssymb, amsthm}
\usepackage{graphicx}
\usepackage{geometry}
\usepackage{tikz}
\usepackage[utf8]{inputenc}
\usepackage{cite}
\usepackage{subcaption}
\usepackage[normalem]{ulem}
\usepackage{hyperref}
\usepackage{pgfplots}
\pgfplotsset{compat=1.16} 
\usepackage{color} 
\usepackage{mathtools} 

\usetikzlibrary{
    positioning,    
    fit,            
    shapes.geometric, 
    shapes.misc,    
    decorations.pathmorphing 
}

\newtheorem{theorem}{Theorem}
\newtheorem{lemma}{Lemma}
\newtheorem{claim}{Claim}

\theoremstyle{definition}
\newtheorem{definition}{Definition}
\newtheorem{problem}{Problem}

\newcommand{\zd}{z^\dagger}
\newcommand{\psg}{$(p+1)K_2$-saturated graph}
\newcommand{\psgs}{$(p+1)K_2$-saturated graphs}

\newcommand{\defi}{\mathrm{def}}
\newcommand{\zf}{\left\lfloor \zd \right\rfloor}
\newcommand{\zc}{\left\lceil \zd \right\rceil}
\newcommand{\mc}{\mathcal{C}}
\newcommand{\zt}{\tilde z}
\newcommand{\sat}{\mathrm{sat}}
\newcommand{\Sat}{\mathrm{Sat}}
\newcommand{\SQRT}{\sqrt{8n+8k+1}}
\newcommand{\dL}{\delta_L}
\newcommand{\dR}{\delta_R}

\title{The Minimum Number of Edges in $(p+1)K_2$-Saturated Graphs}
\author{%
  Xiaoteng Zhou\thanks{Corresponding author. Email: \texttt{zxt@amp.i.kyoto-u.ac.jp}.}\\
  Graduate School of Informatics, Kyoto University, Japan\\
  \texttt{zxt@amp.i.kyoto-u.ac.jp}
  \and
  Kazuya Haraguchi\\
  Graduate School of Informatics, Kyoto University, Japan\\
  \texttt{haraguchi@amp.i.kyoto-u.ac.jp}
  \and
  Hanchun Yuan\\
  Zhejiang Normal University, China\\
  \texttt{hanchunyuan@zjnu.edu.cn}
}

\begin{document}

\maketitle

\begin{abstract}
Given a family of graphs $\mathcal{F}$, 
a graph $G$ is $\mathcal{F}$-saturated if it is $\mathcal{F}$-free but the addition of any missing edge creates a copy of some $F \in \mathcal{F}$. 
The study of the minimum number of edges in  $\mathcal{F}$-saturated graphs is a central topic in extremal graph theory.

Let $(p+1)K_2$ denote a matching of size $p+1$. 
Determining the minimum number of edges in a $(p+1)K_{2}$-saturated graph is a fundamental question in this area,
explicitly posed as Problem~9 in the survey by Faudree et al. (2011). 
In this paper, we refine the structural analysis of \psgs{} and derive an explicit formula for the number of edges in terms of a single integer parameter. 
By minimizing this formula we determine $\sat(n,(p+1)K_2)$ for all $n>2p$, 
thereby resolving Problem~9 in full generality and extending earlier results of K\'aszonyi--Tuza (1986) and Zhang--Lu--Yu (2023).  
Moreover, by maximizing the same formula we recover the classical Erd\H{o}s--Gallai (1959) upper bound on the number of edges in such graphs.
\end{abstract}

\section{Introduction}
\label{intro}
Let $\mathcal{F}$ be a family of graphs. 
A graph $G = (V,E)$ is called $\mathcal{F}$-saturated 
if no $F \in \mathcal{F}$ is a subgraph of $G$,
but for any edge $e \in E(\overline{G})$, $G+e$ contains some $F \in \mathcal{F}$ as a subgraph, 
where $\overline{G}$ is the complement of $G$.

For a positive integer $n$, the \emph{saturation number} is defined as 
\[
\sat(n,\mathcal{F}) \triangleq \min\{|E(G)| : |V(G)|=n,\, G \text{ is  $\mathcal{F}$-saturated}\}.
\]
Similarly,  $\Sat(n,\mathcal{F})$ denotes the family of 
$n$-vertex $\mathcal{F}$-saturated graphs with exactly $\sat(n,\mathcal{F})$ edges.

In 1941, Turán~\cite{Turan1941} introduced the idea of an extremal number 
and determined the maximum number of edges in $K_r$-free graphs. 
Later, Erd\H{o}s,
Hajnal and Moon~\cite{AProbleminGraphTheory} 
focused on the corresponding minimization problem.

The systematic study of $\sat(n,\mathcal{F})$ was surveyed by 
Faudree et al.~\cite{FFS2011}, who compiled a wide range of results, conjectures, 
and open problems.
Let $(p+1)K_2$ denote a matching of size $p+1$. 
Determining $\sat(n,(p+1)K_2)$, the minimum number of edges 
in a \psg{} of order $n$, 
is an important  open problem
(see Problem~9 of~\cite{FFS2011}).
In this context, it is natural to assume $n>2p$
since a graph with no more than $2p$ vertices
cannot contain a matching of size $p+1$
even when missing edges are added somehow. 

Several partial results are known. 
Mader~\cite{Mader1973} proved a structural theorem for $(p+1)K_{2}$-saturated graphs.
He showed that such a graph is either a disjoint union of odd-order cliques, 
or it is connected and contains at least one vertex of degree $|V(G)|-1$.
This is summarized in the following theorem.
\begin{theorem}[Mader~\cite{Mader1973}]\label{thm:Mader}
Let $G$ be a $(p+1)K_{2}$-saturated graph on $n$ vertices. 
If $G$ is disconnected, then $G$ is a disjoint union of cliques, each of which has an odd number of vertices. 
If $G$ is connected, then $G$ contains a vertex of degree $n-1$, 
and the deletion of this vertex 
yields a $pK_{2}$-saturated graph.
\end{theorem}

The saturation number $\sat(n,(p+1)K_{2})$ for $(p+1)K_2$
and graphs $\Sat(n,(p+1)K_{2})$
achieving the number  
are already determined when $n$ is sufficiently larger
than $p$. To be more precise,
K\'aszonyi and Tuza~\cite{KT1986} proved the following theorem. 
\begin{theorem}[K\'aszonyi and Tuza~\cite{KT1986}]
  \label{thm:KT}
  For positive integers $n$ and $p$,
  if $n\ge 3p$, then 
\[
\sat(n,(p+1)K_{2}) = 3p 
\quad \text{and} \quad 
\Sat(n,(p+1)K_{2}) = \{\, pK_{3} \cup (n-3p)K_{1} \,\}.
\]
\end{theorem}

For the case of $2p<n<3p$,
no complete analysis exists in the literature
although there are some related results. 
Let $k = n - 2p$.
Zhang, Lu, and Yu~\cite{ZhangLuYu2023} determined $\sat(n,(p+1)K_2)$ and $\Sat(n,(p+1)K_2)$ for $n > \sqrt{n}+2p$ (i.e., $k > \sqrt{n}$ and $p \ge 2$). 
They showed that in this range every extremal $(p+1)K_2$-saturated graph is disconnected, and all its components are complete graphs of odd order whose sizes differ by at most $2$.
Their result is summarized as follows. 
\begin{theorem}[Zhang, Lu, and Yu~\cite{ZhangLuYu2023}]\label{thm:ZhangLuYu}
  For positive integers $n$ and $p$ such that
  $n\ge 2p\ge 4$,
  let $k=n-2p$.
  If $\sqrt{n} + 2p < n$, or equivallently,
  if $k>\sqrt{n}$,
  then there exist integers $A$ and $B$ such that
  $A+B = k$, $A(2k+1) + B(2k-1) = n$,
\[
\sat(n, (p+1)K_2) = A \binom{2k+1}{2} + B \binom{2k-1}{2}
\]
and
\[
\Sat(n, (p+1)K_2) = \{ AK_{2k+1} \cup BK_{2k-1} \}.
\]
\end{theorem}

\noindent
Furthermore, 
Yuan~\cite{Yuan25} showed that when $n < \sqrt{n/2} + 2p $ (i.e., $k < \sqrt{n/2}$), 
every $(p+1)K_{2}$-saturated graph must contain at least one vertex of degree $|V(G)|-1$.

As described above, however, 
the general problem of determining the minimum number of edges in 
$(p+1)K_2$-saturated graphs for all values of $n$ and $p$ had remained open.
Our approach applies to all $n>2p$,
so it is natural to formulate the problem in full generality as follows.

\begin{problem}\label{pro:open}
Determine $\sat(n,(p+1)K_2)$ and $\Sat(n,(p+1)K_2)$ for all integers $n>2p$ (i.e. $k \ge 1$).
\end{problem}
In this paper,
we provide a complete solution to Problem~\ref{pro:open}.
We identify a sharp boundary between
two behaviors (i.e., connected or disconnected)
of extremal $(p+1)K_2$-saturated graphs. 
For each case,
we give an explicit formula
for the minimum number of edges as a function of $n$ and $k$. 

This paper is organized as follows.
Section~2 introduces notation and 
provides preliminary results that are used
in later proofs.
Section~\ref{sec:estimate} gives Theorem~\ref{thm:min-edge},
the main theorem of this paper,
where we derive a unified expression for $\sat(n,(p+1)K_2)$.
Section~\ref{sec:conseq} provides consequences
that are obtained from the main theorem,
including how Theorem~\ref{thm:min-edge} solves
Problem~\ref{pro:open}
and how it generalizes Theorems~\ref{thm:KT} and \ref{thm:ZhangLuYu}. 
Furthermore, our framework recovers the classical upper bound on the number of edges of such graphs
that is derived by Erd\H{o}s and Gallai~\cite{ErdosGallai1959}.
Section~\ref{sec:conclusion} gives concluding remarks.

\section{Preliminaries}
We work with
standard terminologies of graph theory~\cite{Diestel2025}.
All graphs considered in this paper are simple and undirected.
For $v\in V(G)$, we denote by $d(v)$ the degree of $v$. 
For any vertex set $U \subseteq V(G)$, 
we write $G-U$ for the subgraph of $G$ induced by $V(G)\backslash U$.
$\mc(G)$ denotes the set of connected components of $G$.
We denote by \(\nu(G)\) the matching number
(i.e., the size of a maximum matching) in \(G\).

\begin{definition}[universal vertex] 
A vertex $v$ in a graph $G$ is called a \emph{universal vertex} if it is adjacent to every other vertex in $G$, 
i.e., its degree is $|V(G)| - 1$. 
\end{definition}

\begin{definition}[deficiency]
The \emph{deficiency} of a graph $G$ is defined as
$\defi(G) = |V(G)| - 2\nu(G)$.
\end{definition}

Let $G$ be a $(p+1)K_2$-saturated graph with $n$ vertices. 
We have $\nu(G)=p$ by defintion, and
for convenience, we denote its deficiency by $k$,
that is, $k=n-2p$, unless no confusion arises. 
%
%
%
K\'aszonyi and Tuza~\cite{KT1986} refined Mader's decomposition (i.e., Theorem~\ref{thm:Mader})
by explicitly analyzing the graph
obtained after deleting all universal vertices from $G$, 
thereby giving a clearer structural description.
\begin{itemize}
\item 
  If $G$ is not connected,
  then $G$ is a disjoint union of complete graphs,
  each of odd order. 
\item If $G$ is connected, then
  $G$ contains a nonempty set
  $Z$ of universal vertices,
  and $G - Z$ is a disjoint union of complete graphs,
  each of odd order.
  Furthermore,
  each complete subgraph $C_i$ in $G-Z$ contributes
  to the maximum matching so that
  $|V(C_i)| = 2\nu(C_i) + 1$,
  where we may write $p_i = \nu(C_i)$ unless
  no confusion arises. 
\end{itemize}

We provide preliminary results
that are used in later proofs.
Let $G$ be a $(p+1)K_2$-saturated graph.
As described above,  
there is a decomposition $V(G)=Z\cup V(G-Z)$
such that $Z$ is the set of universal vertices and
$G-Z$ is a disjoint union of odd cliques. 
This yields the relation among $|\mathcal{C}(G-Z)|$,
$k$ and $|Z|$.

\begin{lemma}\label{lem:universal_vertex_matching}
Let $G$ be a $(p+1)K_2$-saturated graph with $k=n-2p\ge 1$, 
and $Z$ be the set of its universal vertices. 
Then $\nu(G-Z) = \nu(G) - |Z|$ holds. 
\end{lemma}
\begin{proof}
  If $G$ is disconnected, then $Z = \emptyset$ and the statement holds trivially.
  If $G$ is connected, then by Theorem~\ref{thm:Mader},
  $Z \ne \emptyset$, 
  and deleting any vertex $v \in Z$ yields a $pK_2$-saturated graph $G - \{v\}$, that is, $\nu(G - \{v\}) = p-1 = \nu(G) - 1$.
Moreover, if $Z \setminus \{v\} \ne \emptyset$, then any other universal vertex $v' \in Z \setminus \{v\}$ remains a universal vertex in $G - \{v\}$ (since $d_{G-\{v\}}(v') = |V(G)| - 2$). 
By repeatedly deleting all vertices in $Z$ one by one, we obtain $\nu(G-Z) = \nu(G) - |Z|$. 
\end{proof}

Using Lemma~\ref{lem:universal_vertex_matching},
we obtain the following lemma
that enables us to analyze the structure of $G$
and its number of edges based on the parameter $|Z|$,
since for a $(p+1)K_2$-saturated graph $G$ with
a given deficiency $k$, 
the number of components of $G-Z$ is determined by $|Z|$.

\begin{lemma} \label{thm:components_deficiency_relation}
Let $G$ be a $(p+1)K_2$-saturated graph with $k \ge 1$, 
$Z$ be the set of its universal vertices. 
The number of components of the subgraph $G-Z$
is given by $|\mathcal{C}(G-Z)| = k + |Z|$.
\end{lemma}
\begin{proof}
In $G-Z$ there are $|\mathcal{C}(G-Z)|$ components, each of odd order. Thus $\defi(G-Z) = |\mathcal{C}(G-Z)|$.
By the definition of deficiency and
Lemma~\ref{lem:universal_vertex_matching}, 
\[
\begin{aligned}
  |\mc(G-Z)| &= \defi(G-Z) 
  = |V(G-Z)| - 2\nu(G-Z) 
  = |V(G)| - |Z| - 2(\nu(G) - |Z|) 
  = \defi(G) + |Z|\\
  &= k+|Z|.
\end{aligned}
\]  
This completes the proof.
\end{proof}

\section{Main Theorem}\label{sec:estimate}

The problem of determining 
$\sat(n,(p+1)K_{2})$ is regarded as a discrete optimization
problem that asks to minimize the number of edges
under constraints 
such that there are $n$ vertices;
the matching number is $p$; and
the matching number increases by adding any missing edge
to the graph.
After making preparatory analyses on the optimal value
(i.e., the minimum number of edges in
a $(p+1)K_2$-saturated graph) in Section~3.1,
we state Theorem~\ref{thm:min-edge}, our main theorem,
and its proof in Section~3.2. 


\subsection{Analyses on Number of Edges}
Recall that we are given the number $n$ of vertices
and the matching number $p$, and
we denote by $k=n-2p$ the deficiency. 
For a $(p+1)K_2$-saturated graph $G$ with $n$ vertices,
let $Z$ denote the set of universal vertices,
where we let $z\coloneqq|Z|$ for simplicity.
By Lemma~\ref{thm:components_deficiency_relation},
the number of components in $G-Z$ is
$|\mathcal{C}(G-Z)|=k+z$. 
We denote by $C_1,C_2,\dots,C_{k+z}$ the components
in $G-Z$ and by
$p_1,p_2,\dots,p_{k+z}$ their matching numbers.
Each $C_i$, $i=1,2,\dots,k+z$ has an odd number of vertices
and hence $|V(C_i)|=2p_i+1$ holds. 
Then the number $|E(G)|$ of edges is represented as follows. 


\begin{align}
  2|E(G)| &= \sum_{v \in V(G)} d(v) =
  \sum_{v \in Z} d(v) + \sum_{v \in V(G-Z)}d(v) 
  = z(n - 1)+ \sum_{i=1}^{|\mc(G-Z)|} (2p_i + 1)(2p_i + z)\nonumber\\
  &= z(n - 1) + \sum_{i=1}^{k+z} \big(4p_i^2 + 2p_i z + 2p_i + z\big). \label{eq:sum2}
\end{align}

On the other hand, the number $n=|V(G)|$ 
of vertices in $G$ yields the relation 
$n=z + \sum_{i=1}^{k+z}(2p_i + 1)$, 
which simplifies to
\begin{equation}\label{eq:sum_si}
  \sum_{i=1}^{k+z} p_i = \frac{n-k}{2} - z.
\end{equation}
Substituting~\eqref{eq:sum_si} into~\eqref{eq:sum2},
we have
\begin{equation}\label{eq:simp_sum}
  2|E(G)|=
  -z^2 + (2n-3)z + n - k + 4\sum_{i=1}^{k+z} p_i^2.
\end{equation}
In \eqref{eq:simp_sum}, $n$ and $k$ are constants.
If $|E(G)|=\sat(n,(p+1)K_2)$,
then the right hand of \eqref{eq:simp_sum}
should be minimized over $z$
and $p_1,p_2,\dots,p_{k+z}$. 

\begin{claim}\label{claim:popt}
  For a fixed $z$,
  $\sum_{i=1}^{k+z} p_i^2$ is minimized
  if $|p_i-p_j|\le1$ for all $i,j=1,2,\dots,k+z$. 
\end{claim}
\begin{proof}
  If there exist indices $i,j$ with $p_i \ge p_j + 2$,
  then replacing
  $(p_i,p_j)$ by $(p_i-1,p_j+1)$ preserves $\sum_{i=1}^{k+z} p_i$
  but decreases
  $p_i^2 + p_j^2$ by
  \[
(p_i-1)^2 + (p_j+1)^2 - (p_i^2 + p_j^2) = -2(p_i - p_j - 1) < 0,
  \]
  contradicting the minimality.
\end{proof}

For a fixed $z$, let $S_z \coloneqq \sum_{i=1}^{k+z} p_i$,
$a\coloneqq\lfloor S_z/(k+z) \rfloor$,
and $b\coloneqq S_z-a(k+z)$, that is,
$a$ and $b$ are the quotient and the residue
of $S_z\div(k+z)$.
By Claim~\ref{claim:popt},
the sum $\sum_{i=1}^{k+z}p^2_i$ is minimized 
when there are $b$ components with $p_i=a+1$
and $k+z-b$ components with $p_i=a$. 
We see that
\[\min_{\{p_i\}} \sum_{i=1}^{k+z} p_i^2 = b(a+1)^2 + (k+z-b)a^2=(k+z)a^2 + 2ba + b.\]
By \eqref{eq:sum_si}, we have
$S_z=a(k+z) + b=\frac{n-k}{2}-z$. Eliminating $a$, we have
\begin{equation}\label{eq:min_si2}
  \begin{aligned}
  \min_{\{p_i\}} \sum_{i=1}^{k+z} p_i^2 
  &= (k+z)(\frac{\frac{n-k}{2}-z-b}{k+z})^2 + 2b \frac{\frac{n-k}{2}-z-b}{k+z} + b  \\
  &= \frac{(n+k)^2}{4(k+z)} - n + z - \frac{b^2}{k+z} + b.
\end{aligned}
\end{equation}

We have seen the minimum of
$\sum_{i=1}^{k+z} p_i^2$ in
$2|E(G)|$ (and hence $|E(G)|$) 
for a fixed $z$. Regarding $2|E(G)|$ as a function of $z$,
we define $D(z)$ to be
\begin{align*}
  D(z) &\triangleq -z^2 + (2n-3)z + n - k + 4\min_{\{p_i\}} \sum_{i=1}^{k+z} p_i^2\\
  &=-z^2 + (2n+1)z + \frac{(n+k)^2}{k+z} - 3n-k + 4(b - \frac{b^2}{k+z}),
\end{align*}
where the second equality is due to \eqref{eq:min_si2}
and $-3n-k$ are constants. 
By finding the integer $z$ that  minimizes $D(z)$, 
we determine $\sat(n,(p+1)K_2)$.
We denote by $z^\ast$ the optimal solution,
that is, $D(z^\ast)=\sat(n,(p+1)K_2)$. 

Let us decompose $D(z)$ 
into two functions further. 
Let 
\[g(z) \coloneqq -z^2 + (2n+1)z + \frac{(n+k)^2}{k+z}\ \ \textrm{and}\ \
R(z) \coloneqq 4\big(b - \frac{b^2}{k+z}\big).\]
We refer to $g(z)$ as the main term and to $R(z)$ as the remainder term.
We will later prove that the main term
$g(z)$ dominates the remainder term $R(z)$ in some sense,
and so the behavior of $g(z)$ is particularly important.
We regard $D(z)$, $g(z)$ and $R(z)$ as functions
over the real numbers \(\mathbb{R}\), 
but in the end we only focus on integer values of $z$ in the interval 
$[0, \tfrac{n-k}{2}]$
since $S_z = \frac{n-k}{2}-z \ge 0$.

Note that 
$b$ is an integer from $\{0,1,\dots,k+z-1\}$
by definition. 
The function $x \mapsto x - x^2/(k+z)$ attains its maximum at $x = (k+z)/2$, 
which equals $(k+z)/4$.
Hence
we have
$0 \le R(z) \le k+z$.

\paragraph{Minimizer of the main term $g(z)$.}
The first derivative of the main term $g(z)$ is:
\[
\begin{aligned}
  g'( z) = -2z + 2n + 1 - \frac{(n+k)^2}{(z+k)^2},
\end{aligned}
\]
and the second derivative is:
\[g''(z) = -2 + 2 \frac{(n+k)^2}{(k+z)^3}.\]

Observe that $g''(z)$ decreases on $(-\infty, -k) \text{ and on } (-k, +\infty)$. 
$g''(z) = 0$ has exactly one solution, 
denoted by $z_0 = (n+k)^{2/3} - k$.
Therefore, \( g'(z) \) increases in the interval \( (-k, z_0) \), and decreases on \( (z_0, +\infty) \).

For $z \ne -k$, the equation $g'(z) = 0$,
which is equivalent to 
$(n-z)(2z^2 + (4k-1)z + 2k^2 -n-2k) = 0$,
has three roots:
\[
z_1 = \frac{1}{4} \left( -\sqrt{8k + 8n + 1 } - 4k + 1 \right),
\quad
z_2 = \frac{1}{4} \left( \sqrt{8k + 8n + 1 } - 4k + 1 \right),
\quad 
z_3 = n. 
\]
Among these $z_2$ is important 
and hence we let $\zd\coloneqq z_2$. 
If $k \le \sqrt{n},$ then it holds that 
\[z_1 < z_2 < z_3 ,\quad z_1 < z_0, \quad z_1 < 0.\]

\begin{claim}\label{claim:g-monotone}
  If $\zd > 0$, then
  $g(z)$ is strictly decreasing on $(0, \zd)$
  and strictly increasing on $(\zd, n)$.
  Otherwise (i.e., if $\zd\le0$), 
  $g(z)$ is strictly increasing on $(0, n)$.
\end{claim}

\subsection{Main Theorem and Its Proof}
Now we are ready to present the main theorem. 
\begin{theorem} \label{thm:min-edge}
  For positive integers $n$ and $p$ such that $n>2p$,
  let $k = n - 2p$. Then we have
\begin{align}
\sat(n,(p+1)K_2) =
\left\{
\begin{array}{ll}
  \displaystyle \min_{z \in \{\zf, \zc\}} \dfrac{D(z)}{2} & \text{if } n > 2k^2 - 2k,\\
  \dfrac{D(0)}{2}  & \text{if } n \le 2k^2 - 2k. 
\end{array}
\right.
\label{eq:min-edge}
\end{align}
\end{theorem}
\noindent
The theorem tells that,
if $n\le 2k^2-2k$, then $z^\ast=0$, and otherwise,
$z^\ast\in\{\zf,\zc\}$. 


Let us describe the strategy of the proof.
The condition whether $n > 2k^2 - 2k$
or $n \le 2k^2 - 2k$ is equivalent to whether
$\zd>0$ or $\zd\le 0$ as follows;
Recall the definition of $\zd$.  
\[
\zd = \frac{1}{4}\left( \SQRT - 4k + 1 \right),
\]
and $\zd > 0$ means $\SQRT > 4k-1$,
where $4k-1>0$ holds by $k\ge1$. 
Squaring both sides of the inequality gives
\[
8n+8k+1 > (4k-1)^2 = 16k^2 - 8k + 1
\;\Longleftrightarrow\;
n > 2k^2 - 2k.
\]
By Claim~\ref{claim:g-monotone},
if $\zd>0$, then the main term $g(z)$ is minimized
at $z=\zd$, and otherwise, $g(z)$ is minimized at $z=0$. 
The proof for Theorem~\ref{thm:min-edge}
observes the following four cases:
Combining these results,
we show the theorem.

\begin{itemize}
\item {\bf Case (I):}
  $n\ge 19$ and $\zd>0$ (i.e., $n>2k^2-2k$).
  We show in Lemma~\ref{lemma:z_optimal_equivalence}
  that the integer minimizer
  $z^\ast$ for $D(z)$ is around the minimizer $\zd$ for $g(z)$;
  to be more precise,
  $z^\ast\in\{\lfloor \zd \rfloor, \lceil \zd \rceil\}$ holds.
\item {\bf Case (II):}
  $\zd\le0$ and $k\le\sqrt{n}$
  (i.e., $k^2\le n\le 2k^2-2k$).
  We show in Lemma~\ref{lemma:z2Les0} that 
  $z^\ast=0$ holds.
\item {\bf Case (III):}
  $n\le18$ and $k\le\sqrt{n}$.
  We show in Lemma~\ref{lemma:enum} that
  all possible $(n,k)$ satisfy
  the statement of Theorem~\ref{thm:min-edge}
  by enumeration.
\item {\bf Case (IV):}
  The remaining case such that
  $\zd\le0$ and $k>\sqrt{n}$
  (i.e., $n<k^2\le 2k^2-2k$).
  This case is already covered by 
  Theorem~\ref{thm:ZhangLuYu}.
\end{itemize}
%


The remainder of this section is devoted to
our observation on Cases (I) to (IV),
followed by the proof for Theorem~\ref{thm:min-edge}. 

\begin{lemma}[Case (I)]\label{lemma:z_optimal_equivalence}
  For positive integers $n$ and $p$
  such that $n\ge19$ and $n>2p$, 
  if $\zd>0$, then $z^\ast\in\{\zf,\zc\}$ holds. 
\end{lemma}
%
\begin{proof}
  The proof analyzes the behavior of $D(z)$ around $\zd$.
  We show that the function the main term $g(z)$
  is convex in the relevant interval and that the "jump" 
  $g(z+1) - g(z)$ or $g(z) - g(z-1)$ when moving away from $\zd$ is large enough to dominate the remainder term $R(z)$.
  This requires a detailed algebraic computation of $g(\zd + 1) - g(\zd)$ and $g(\zd - 1) - g(\zd)$ to show they are sufficiently large.

  Let $k=n-2p$. 
  Let $\zt \in \mathbb{Z}$ be such that
$\zt \notin \{\zf,\zc\}$ and $\zt \in [0,\frac{n-k}{2}]$.
We will prove that, for all $n \ge 19$, such a $\zt$ cannot minimize $D$.
This will imply that any minimizer $z^*$ of $D$ must belong to $\{\zf,\zc\}$.

We write $\zt$ in terms of its distance from $\zd$.
We distinguish two cases:
\[
\begin{cases}
\text{if } \zt \ge \zc + 1, & 
\zt = \zd + T + \dL, \quad \dL = \zc - \zd,\\[4pt]
\text{if } \zt \le \zf - 1, &
\zt = \zd - T - \dR, \quad \dR = \zd - \zf,
\end{cases}
\]
where $T \ge 1$ is the integer offset from the nearest of $\zf$ or $\zc$.

We will show that
\[
\begin{aligned}
  &D(\zt)  = g(\zd + T + \dL) + R(\zd + T + \dL) > D(\zc) = g(\zc) + R(\zc)
  &&\text{if } \zt \ge \zc,\\
  &D(\zt) = g(\zd - T - \dR) + R(\zd - T - \dR) > D(\zf) = g(\zf) + R(\zf)
  &&\text{if } \zt \le \zf,
\end{aligned}
\]
it suffices to prove
\[
\begin{aligned}
  &g(\zd + T + \dL) > g(\zc) + R(\zc)
  &&\text{if } \zt \ge \zc,\\
  &g(\zd - T - \dR) > g(\zf) + R(\zf)
  &&\text{if } \zt \le \zf.
\end{aligned}
\]

Recall that by Claim~\ref{claim:g-monotone}, $g(z)$ is strictly decreasing on $(0,\zd)$ and strictly increasing on $(\zd,n)$.
Therefore, for $T \ge 1$,
\[
\begin{aligned}
g(\zd + T + \dL) &\ge g(\zd + 1 + \dL),\\
g(\zd - T - \dR) &\ge g(\zd - 1 - \dR).
\end{aligned}
\]
Hence it is enough to prove the following two inequalities:
\[
\begin{aligned}
  &g(\zd + 1 + \dL) > g(\zc) + R(\zc),
  &&\text{if } \zt \ge \zc,\\
  &g(\zd - 1 - \dR) > g(\zf) + R(\zf),
  &&\text{if } \zt \le \zf.
\end{aligned}
\]

To estimate these increments, 
we claim that $\zc+1 < z_0$ for $n \ge 9$ 
where $z_0$ is the unique real solution of $g''(z)=0$.
Since $g''(z)>0$ on $(-k,z_0)$, this implies $g$ is convex on $[\zf-1,\zc+1]$ for all $n\ge 9$.
Recall that
\[
\zd = \frac{\SQRT - 4k + 1}{4}
\]
and
\[
z_0 = (n+k)^{2/3} - k.
\]
Since $\zc = \zd + \dL < \zd + 1$, 
it suffices to show 
\[
\zd + 2 < (n+k)^{2/3} - k,
\]
or, rearranging,
\begin{equation}\label{ineq:convex}
  (n+k)^{2/3} - \frac{\SQRT}{4} > \frac{9}{4}.
\end{equation}

Set
\[
f(n,k) = (n+k)^{2/3} - \frac{\SQRT}{4}.
\]
Note that $f$ depends on $n$ and $k$ only through $n+k$, so $n$ and $k$ play a symmetric role.
Moreover,
\[
\frac{\partial f}{\partial k}
= \frac{2}{3 (n+k)^{1/3}}
  - \frac{1}{\sqrt{8n+8k+1}}
> 0
\quad\text{for all } n,k \ge 1.
\]
Thus $f(n,k)$ is strictly increasing in $k$. Since we are assuming $k \ge 1$, we have
\[
f(n,k) \ge f(n,1)
= (n+1)^{2/3} - \frac{\sqrt{8n+9}}{4}.
\]
The right-hand side is increasing in $n$, 
and at $n=9$ this value is approximately $2.39$:
\[
(10)^{2/3} - \frac{\sqrt{81}}{4}
\approx 2.39 > \frac{9}{4}.
\]
Hence for all $n \ge 9$ and $k \ge 1$, $f(n,k) > 9/4$,
then we obtain that~\eqref{ineq:convex} holds, thereby  $\zc+1 < z_0$.

Since $\zc+1 < z_0$, we have $g''(z) > 0$ throughout the interval $(\zf-1,\zc+1)$,
and $g'(\zd)=0$.
By the Mean Value Theorem, 
there exist $c_1 \in (\zd + 1, \zd + 1 + \dL), c_2 \in (\zd, \zd + \dL)$ such that 
\[g(\zd + 1 + \dL) - g(\zd + 1) = g'(c_1)\dL \ge g'(c_2)\dL = g(\zd + \dL) - g(\zd),\]
hence
\[g(\zd + 1 + \dL) - g(\zd + \dL) \ge g(\zd + 1) - g(\zd).\] 
Similarly
\[g(\zd - 1) - g(\zd - 1 - \dR)  \le g(\zd) - g(\zd - \dR),\]
which is 
\[ g(\zd - 1 - \dR) -g(\zd - \dR) \ge g(\zd - 1) -  g(\zd). \]
We now compute the one-step differences of $g$:
{\small
\begin{align}
  g(\zd+1)-g(\zd)
  &= 2n - 2\zd - \frac{(n+k)^2}{(k+\zd)(k+\zd+1)} \nonumber\\
  &= 2n + 2k - \frac{16(n+k)^2}{(\SQRT+1)(\SQRT+5)} - \frac{\SQRT+1}{2} \nonumber\\
  &= 2(n+k) - \frac{16(n+k)^2(\SQRT-1)}{(8n+8k)(\SQRT+5)} - \frac{\SQRT+1}{2} \nonumber\\
  &= 2(n+k)\frac{6}{\SQRT+5} - \frac{\SQRT+1}{2} \nonumber\\
  &= 2(n+k)\frac{6(\SQRT-5)}{8n+8k-24} - \frac{\SQRT+1}{2} \nonumber\\
  &= (n+k-3+3)\frac{3(\SQRT-5)}{2n+2k-6} - \frac{\SQRT+1}{2} \nonumber\\
  &= \SQRT - 8 + \frac{9(\SQRT-5)}{2(n+k-3)} \nonumber\\
  &> \SQRT - 8, \nonumber
\end{align}
}
and
{\small
\[
\begin{aligned}
  g(\zd - 1) - g(\zd) &= 2\zd - 2n - 2 + \frac{(n+k)^2}{(k+\zd)(k+\zd-1)}\\
  &= -2k -2n -2 + \dfrac{16(n+k)^2}{(\SQRT + 1)(\SQRT -3)} + \dfrac{\SQRT + 1}{2} \\
  &=(n+k)\big( \dfrac{2(\SQRT - 1)}{\SQRT - 3} - 2 \big) - 2  + \dfrac{\SQRT + 1}{2}  \\
  &= (n+k) \dfrac{4(\SQRT + 3)}{8n+8k-8} - 2 + \dfrac{\SQRT + 1}{2}\\
  &= \sqrt{8n+8k+1} + \frac{\sqrt{8n+8k+1}+3}{2n+2k-2}\\
  &> \sqrt{8n+8k+1}.
\end{aligned}
\]
}
Hence we obtain 
\begin{equation}\label{inq:delta1}
  g(\zd + 1) - g(\zd) > \SQRT - 8,
\end{equation}
\begin{equation}\label{inq:delta2}
  g(\zd-1) - g(\zd) > \SQRT.
\end{equation}
Comparing $D(\zd+\dL+T)$ with $D(\zd+\dL)$,
we have
\[ \begin{aligned} D(\zd+\dL + T) - D(\zd+\dL) 
  &=g(\zd +\dL + T) - g(\zd+\dL) + R(\zd+\dL + T) - R(\zd+\dL)\\ 
  &\ge g(\zd +\dL + 1) - g(\zd+\dL) + R(\zd+\dL + T) - R(\zd+\dL) \\ 
  &\ge g(\zd +\dL + 1) - g(\zd+\dL) - R(\zd+\dL)\\ 
  &\ge g(\zd + 1) - g(\zd) - R(\zd+\dL)\\ 
  &\ge g(\zd + 1) - g(\zd) - (k+\zd + 1),
 \end{aligned} \] 
since $k + \zd + 1 = \frac{\SQRT + 1}{4} + 1 $, 
by~\eqref{inq:delta1},
\[ \begin{aligned} 
  D(\zd+\dL + T) - D(\zd+\dL) 
  &\ge g(\zd + 1) - g(\zd) - \frac{\SQRT + 5}{4}\\ 
  & >\SQRT - 8- \frac{\SQRT + 5}{4}\\ 
  &= \frac{3\SQRT}{4} - 9.25\\
  &> \frac{3\sqrt{8n+1}}{4} - 9.25.
\end{aligned} \] 
Note that $\frac{3\sqrt{8\times 19+1}}{4} - 9.25 \approx 0.027 > 0,$ 
we obtain $D(\zd + \dL + T) > D(\zd + \dL)$ for $n \ge 19$. 

Then for $D(\zd - T -\dR)$ and $D(\zd - \dR)$,
similarly, 
\[ \begin{aligned} D(\zd-T-\dR) - D(\zd-\dR) 
  &= g(\zd-T-\dR) - g(\zd-\dR) + R( \zd -T - \dR) - R(\zd - \dR)\\ 
  &\ge g(\zd-1-\dR) - g(\zd-\dR) + R( \zd -T - \dR) - R(\zd - \dR) \\ 
  &\ge g(\zd-1-\dR) - g(\zd-\dR) - R(\zd - \dR)\\
  & \ge g(\zd-1) - g(\zd) - R(\zd - \dR) \\ 
  &\ge g(\zd - 1) - g(\zd) -(k+\zd - \dR)\\ 
  &\ge g(\zd - 1) - g(\zd) -(k+\zd ) \\ 
  &> \SQRT - \dfrac{\SQRT + 1}{4} \\ 
  &> \frac{3\sqrt{8n+1}}{4} - \frac{1}{4}\\
  & > 0
\end{aligned} \]
for all $n \ge 19$.

Combining the two directions, we conclude that for all $n \ge 19$ and $k \ge 1$,
no $\zt \notin \{\zf,\zc\}$ can give the minimum value of $D$.
Hence any minimizer $z^*$ of $D$ must satisfy
\[
z^* \in \{\zf,\zc\},
\]
as claimed.
\end{proof}


\begin{lemma}[Case (II)]\label{lemma:z2Les0}
  For positive integers $n$ and $p$ such that $n>2p$,
  let $k=n-2p$. 
  If $\zd \le 0$ and $k\le\sqrt{n}$
  (or equivalently, if $k^2\le n \le 2k^2 - 2k$), 
  then $z^* = 0$. 
\end{lemma}
\begin{proof}
Since $k \le \sqrt{n}$ implies $k^2 \le n$, the condition is $k^2 \le n \le 2k^2 - 2k$.
We compare $D(z)$ with $D(0)$ for any integer $z \ge 1$:
\[
\begin{aligned}
  D(z) - D(0) &= g(z) + R(z) - g(0) - R(0) \\
  & \ge g(z) - g(0) - R(0),
\end{aligned}
\]
when $\zd \le 0$, by Claim~\ref{claim:g-monotone}, 
the function $g(z)$ is increasing for $z \in [0, n]$, then
\[
\begin{aligned}
  D(z) - D(0)  & \ge g(z) - g(0) - R(0), \\
  & \ge g(1) - g(0) - R(0)\\
  & \ge 2n+ \frac{(n+k)^2}{k+1} - \frac{(n+k)^2}{k} - k\\
  & = 2n -k - \frac{(n+k)^2}{k(k+1)}.
\end{aligned}
\]
Define $H(n) = 2n -k - \frac{(n+k)^2}{k(k+1)}$, for fixed $k$, $H'(n)=\frac{2(k^2-n)}{k(k+1)}\le0$ for $n \ge k^2$,
hence $H$ is decreasing on $[k^2,\,2k^2 - 2k]$. 
At the right endpoint,
\[
H(2k^2-2k)=\frac{3k(k-2)}{k+1}.
\]
Notice that $n \le 2k^2 - 2k \Rightarrow k \ge \sqrt{\frac{2n+1}{4}}+\frac 12 \ge \frac{\sqrt 3 + 1}{2}$,
then $k \ge 2$, $H(2k^2-2k) \ge 0$.

Therefore $H(n)\ge0$ for all $n\in[k^2,\,2k^2-2k]$, which implies $D(z)\ge D(0)$ for all $z\ge1$.
Therefore for all integers $z\ge 1$, 
$D(z)-D(0)\ge 0$,
the minimum of $D(z)$ over $z\in[0,\tfrac{n-k}{2}]\cap\mathbb{Z}$ is reached at $z=0$.
  
\end{proof}


\begin{lemma}[Case (III)]\label{lemma:enum}
  For positive integers $n$ and $p$
  such that $2p<n\le18$, let $k=n-2p$.
  If $k\le\sqrt{n}$, then \eqref{eq:min-edge} holds.  
\end{lemma}
\begin{proof}
  We show the evidence in Table~\ref{tab:enum}.
  Since $n-k=2p$,
  we restrict ourselves to $(n,k)$
  such that $n$ and $k$ have the same parity. 
  For each feasible pair $(n,k)$, 
  we determined the minimizer of $D(z)$ by direct computation: for every integer $z$ in the feasible range $0 \le z \le \tfrac{n-k}{2}$ we evaluated $D(z)$ and recorded the minimizing value $z^*$. 
  We observed that, in every case,
  if $\zd>0$ (i.e., $n>2k^2-2k$),
  $z^*$ coincides with one of $\zf$ or $\zc$,
  and if $\zd\le0$ (i.e., $n \le 2k^2-2k$),
  $z^*=0$ holds. 
\end{proof}
  
The program code used for the enumeration in Table~\ref{tab:enum} is available at
\url{https://github.com/zhouxiaoteng2000/p-plus-1-K2-sat-enumeration.}

\begin{table}[t!]
\centering
\caption{Integer optimizers $z^*$ for $D(z)$
  and real optimizers $z^\dagger$ for the main term $g(z)$
  for all $(n,k)$ such that
  $n\in\{4,5,\dots,18\}$ and $k\in\{1,2,\dots,\lfloor\sqrt{n}\rfloor\}$ have the same parity}
\label{tab:enum}

\begin{minipage}[t]{0.48\textwidth}
\centering
\textbf{$4 \le n \le 12$}\\[4pt]
\begin{tabular}{rrrr}
\hline
$n$ & $k$ & $z^*$ & $z^\dagger$ \\
\hline
4&2&0&0.00 \\
5&1&1&1.00 \\
6&2&0&0.27 \\
7&1&1&1.27 \\
8&2&0&0.50 \\
9&1&1&1.50 \\
9&3&0&-0.29 \\
10&2&1&0.71 \\
11&1&2&1.71 \\
11&3&0&-0.09 \\
12&2&1&0.91 \\

\hline
\end{tabular}
\end{minipage}
\begin{minipage}[t]{0.48\textwidth}
\centering
\textbf{$13 \le n \le 18$}\\[4pt]
\begin{tabular}{rrrr}
\hline
$n$ & $k$ & $z^*$ & $z^\dagger$ \\
\hline
13&1&2&1.91 \\
13&3&0&0.09 \\
14&2&1&1.09 \\
15&1&2&2.09 \\
15&3&0&0.26 \\
16&2&1&1.26 \\
16&4&0&-0.58 \\
17&1&2&2.26 \\
17&3&0&0.42 \\
18&2&1&1.42 \\
18&4&0&-0.42 \\

\hline
\end{tabular}
\end{minipage}
\end{table}


For Case (IV),
Theorem~\ref{thm:ZhangLuYu} indicates that when $z=0$ (i.e., no universal vertices), the extremal configuration consists of a disjoint union of $A$ copies of $K_{2t-1}$ and $B = k - A$ copies of $K_{2t+1}$ for an integer $t$,
under the constraints:
\[ 
A + B = k, \qquad B(2t + 1) + A(2t - 1) = n\,.
\] 
The second constraint simplifies to $(t-1)k + B = \frac{n-k}{2}$, yielding 
\[ 
t - 1 = \frac{n-k-2B}{2k}\,. 
\] 
The total number of edges in this graph is: 
\[ 
|E(G)| \;=\; B \binom{2t + 1}{2} + A \binom{2t - 1}{2}\,, 
\] 
so twice the number of edges is 
\begin{equation*}
\begin{aligned}
2\,|E(G)| &= 2B\binom{2t + 1}{2} + 2A \binom{2t - 1}{2} \\
&= 4k(t-1)^2 + 2k(t-1) + 8(t-1)B + 6B \\[3pt]
&= \frac{n^2}{k} - n + 4B - \frac{4B^2}{k}\,. 
\end{aligned}
\end{equation*}
This value is precisely $D(0)$ in our framework:
\[ 
D(0) = \frac{(n+k)^2}{k} - 3n - k + 4\Big(b - \frac{b^2}{k}\Big) = \frac{n^2}{k} - n + 4b - \frac{4b^2}{k}\,,
\] 
where in the last expression we use the fact that $b = B = \big(\frac{n-k}{2}\big) \bmod k$ in the $z=0$ configuration.

We now complete the proof of Theorem~\ref{thm:min-edge}.

\begin{proof}[Proof of Theorem~\ref{thm:min-edge}]
We consider three cases:
\begin{enumerate}
    \item[(1)] $k > \sqrt{n}$. 
    By Theorem~\ref{thm:ZhangLuYu}, the minimum of $D(z)$ is attained at $z=0$. 
    Thus $\sat(n,(p+1)K_2) = D(0)/2$.
    \item[(2)] $k \le \sqrt{n}$ and $n \ge 19$: 
    \begin{enumerate}
      \item[(2.1)] If $n > 2k^2 - 2k$, then by Lemma~\ref{lemma:z_optimal_equivalence}, $z^* = \lfloor \zd \rfloor$ or $z^* = \lceil \zd \rceil$.
      \item[(2.2)] If $n \le 2k^2 - 2k$, then by Lemma~\ref{lemma:z2Les0}, $z^* = 0$.
    \end{enumerate}
    \item[(3)] $n \le 18$. 
      The statement was verified by direct enumeration (Lemma~\ref{lemma:enum}).
\end{enumerate}
This covers all possibilities and completes the proof.
\end{proof}

\section{Consequences of Theorem~\ref{thm:min-edge}}
\label{sec:conseq}

Theorem~\ref{thm:min-edge} covers the entire range of $n$
and hence solves Problem~\ref{pro:open}.
We have already seen that the theorem
is a generalization of
Zhang, Lu, and Yu's Theorem~\ref{thm:ZhangLuYu} in \cite{ZhangLuYu2023};
if $n<k^2\le 2k^2-2k$, then
$\dfrac{D(0)}{2}$ equals to the number $\sat(n, (p+1)K_2)$. 

Let us observe how Theorem~\ref{thm:min-edge}
generalizes K\'aszonyi and Tuza's Theorem~\ref{thm:KT} in \cite{KT1986}. 
In our notation we have $k = n - 2p$, so $p = (n-k)/2$.  
The condition in our main theorem
\[
n \le 2k^2 - 2k
\]
is equivalent to
\[
k \ge \frac{1 + \sqrt{2n+1}}{2}.
\]
K\'aszonyi and Tuza considered the case $n \ge 3p$, that is
\[
n \ge 3 \cdot \frac{n-k}{2}
\quad\Longleftrightarrow\quad
k \ge \frac{n}{3}.
\]
When
\[
\frac{n}{3} > \frac{1 + \sqrt{2n+1}}{2},
\]
which holds for all $n \ge 8$, the range $k \ge n/3$ automatically satisfies
\[
k > \frac{1 + \sqrt{2n+1}}{2},
\]
and hence falls into the case $n \le 2k^2 - 2k$ of our theorem.  
For $n \le 7$, our exhaustive enumeration shows that there is no universal vertex in the extremal graphs in this range either.

Now compute $D(0)/2$ in this regime.  
When $z=0$ we have $b = S \bmod k = \bigl((n-k)/2\bigr) \bmod k$.  
Under the condition $3k \ge n$ we have $(n-k)/2 \le k$, so
\[
b \in \left\{ \frac{n-k}{2},\, 0 \right\}.
\]
Substituting these two possibilities into the expression for $D(0)$ shows that in both cases
\[
\frac{D(0)}{2} = 3p.
\]

In summary, the results of K\'aszonyi--Tuza and Zhang--Lu--Yu can be viewed as giving a range of $k$ (in terms of $n$) in which no universal vertices appear:
\[
\text{K\'aszonyi--Tuza: } k \ge \frac{n}{3}, 
\qquad
\text{Zhang--Lu--Yu: } k \ge \sqrt{n}.
\]
When $k$ lies in these ranges, all extremal graphs are disconnected and have no universal vertex.  
For large $n$, our theorem improves these thresholds and still forces the absence of universal vertices.  
For smaller values of $n$, our bounds on $k$ are weaker than those of K\'aszonyi--Tuza and Zhang--Lu--Yu; however,
when $n\le2k^2-2k$ (i.e., $k \ge \frac{1 + \sqrt{2n+1}}{2}$)
in our theorem, 
there are instances where $\lfloor z^\dagger \rfloor = 0$ and $\lfloor z^\dagger \rfloor$ gives the minimum of $D(z)$, which again guarantees that no universal vertices appear in the extremal graphs.

\paragraph{Maximizing $|E(G)|$.}
Finally, we obtain the classical upper bound
on $\sat(n,(p+1)K_2)$ that was
derived by Erd\H{o}s and Gallai~\cite{ErdosGallai1959}
based on our framework. 


For nonnegative real numbers \( a_1, a_2, \dots, a_n \), it follows:
\[
\sum_{i=1}^{n} a_i^2 \le \left( \sum_{i=1}^{n} a_i \right)^2,
\]
with equality if and only if all but at most one of the \( a_i \) are zero.

We obtain the following. 
\[
\begin{aligned}
  \sum_{v \in V(G) }d(v) &\le -z^2 + (2n-3)z + n - k + 4(\sum_{i=1}^{k+z} p_i)^2 \\
  &= -z^2 + (2n-3)z + n - k + 4\left(\frac{n-k}{2} - z\right)^2 \\
  &= 3z^2 + (4k - 2n - 3)z + \bigl(n^2 - 2kn + k^2 + n - k\bigr).\\
\end{aligned}
\]

By analyzing this quadratic function in $z$, 
we naturally recover the classical upper bound in~\cite{ErdosGallai1959}.

\begin{theorem}[Upper Bound on Edge Count~\cite{ErdosGallai1959}]
 
Let $G$ be a $(p+1)K_2$-saturated graph with $n$ vertices and deficiency $k$, where $n - k = 2p$. Then:
\[
|E(G)| \le 
\begin{cases}
\displaystyle \frac{(n - k)(n - k + 1)}{2}, & \text{if } n > 5k - 6;\\[8pt]
\displaystyle \frac{(n - k)(3n+k-2)}{8} , & \text{if } n \le 5k - 6.
\end{cases}
\]

\end{theorem}

\section{Conclusion}\label{sec:conclusion}
In this paper we resolve Problem~\ref{pro:open} by determining the saturation number
\(\sat(n,(p+1)K_2)\) for all \(n>2p\), not only for the range \(2p<n<3p\).
Our main result, Theorem~\ref{thm:min-edge}, gives an explicit formula via the parameter \(z\),
and it shows that the optimum is attained at \(z\in\{\,\lfloor \zd\rfloor,\lceil \zd\rceil\,\}\) whenever \(n>2k^2-2k\),
while \(z=0\) when \(n\le 2k^2-2k\) (with \(k=n-2p\)).
For the finitely many small cases \(n\le 18\) we complete the analysis by an exhaustive enumeration.

Beyond the matching case \(K_2\), we expect that our result will be useful for the broader problem
of determining \(\sat(n,(p+1)K_r)\) for \(r\ge 3\).

\bibliographystyle{plain} 
\bibliography{TheMinimumNumberofEdgesinmSaturatedGraphs}  

\end{document}